\documentclass[12pt,reqno]{amsart}
\usepackage{amsmath, amsfonts, amssymb, amsthm, cite}
\def\Z{\mathbb Z}

\def\N{\mathbb N}

\def\1{{\bf 1}}

\def\pmod #1{\ ({\rm{mod}}\ #1)}

\def\qbinom #1#2#3{{\genfrac{[}{]}{0pt}{}{#1}{#2}}_{#3}}
\theoremstyle{plain}
\newtheorem{theorem}{Theorem}

\newtheorem{lemma}{Lemma}
\newtheorem{corollary}{Corollary}

\theoremstyle{definition}

\theoremstyle{remark}

\allowdisplaybreaks[1]

\begin{document}

\title{Some $q$-congruences involving central $q$-binomial coefficients}

\begin{abstract}
Suppose that $p$ is an odd prime and $m$ is an integer not divisible
by $p$. Sun and Tauraso [Adv. in Appl. Math., 45(2010), 125--148]
gave $\sum_{k=0}^{n-1}\binom{2k}{k+d}/m^k$ and
$\sum_{k=0}^{n-1}\binom{2k}{k+d}/(km^k)$ modulo $p$ for all $d=0,1,
\ldots n$ and $n= p^a$, where $a$ is a positive integer. In this
paper, we present some $q$-analogues of these congruences in the
cases $m=2, 4$ for any positive integer $n$.
\end{abstract}
\author{He-Xia Ni}
\address{Department of Applied Mathematics, Nanjing Audit University\\Nanjing 211815,
People's Republic of China}
\email{nihexia@yeah.net}


\thanks{The work is supported by the Natural Science Foundation of the Higher Education Institutions of Jiangsu Province (20KJB110023) and the National Natural Science Foundation of China (Nos. 12001279 ).  }

\maketitle

\section{Introduction}
\setcounter{theorem}{0}\setcounter{lemma}{0}\setcounter{equation}{0}
A central binomial coefficient has the form $\binom{2n}{n}$ with
$n\in \N =\{0,1,\ldots\}$.
The well-known Fibonacci sequence ${F}_{n\in\N}$ is defined by
$$
F_0=0, F_1=1, {\rm and}\  F_{n+1}=F_n+F_{n-1} \ {\rm for} \ n=1,2,3,\ldots.
$$
Let $p$ be an odd prime.
Sun and Tauraso
\cite{SunR}  proved for any $m\in \Z$ with $p\nmid m$,
\begin{align}
\sum_{k=0}^{p^a-1}(-1)^k\binom{2k}{k+d}&\equiv (-1)^{d-1+\delta_{p,5}}F_{2(|d|-(\frac{p^a}{5}))}\pmod{p},\label{qb2}\\
\sum_{k=0}^{p^a-1}\frac{\binom{2k}{k+d}}{m^k}&\equiv u_{p^a-|d|}(m-2)\pmod{p},\label{qb3}\\
|d|\sum_{k=0}^{p^a-1}\frac{\binom{2k}{k+d}}{km^{k-1}}&\equiv 2(-1)^d+v_{p^a-|d|}(m-2)\pmod{p}\label{qb4},
\end{align}
where $|d| \in \{0, \ldots p^a\}$ with $a\in \Z^{+}$,  $\delta_{i,j}$ is the Kronecker delta, and the two linear recurrences $\{u_{n}(x)\}_{n\in \N}$ and $\{v_{n}(x)\}_{n\in \N}$ of polynomials are defined as follows:
$$ u_0(x)=0,\ u_1(x)=1,  {\rm and}\ u_{n+1}(x)=xu_n(x)-u_{n-1}(x)\ (n=1,2,\ldots),$$
$$ v_0(x)=2,\ v_1(x)=x,  {\rm and}\ v_{n+1}(x)=xv_n(x)-v_{n-1}(x)\ (n=1,2,\ldots),$$
and
$$\ u_{-1}(x)=xu_0(x)-u_1(x)=-1, v_{-1}(x)=xv_0(x)-v_1(x)=x.$$

Later Sun and Tauraso \cite{SunR1} also proved
\begin{align}
\sum_{k=0}^{p^a-1}\binom{2k}{k+d}&\equiv \left(\frac{p^a-|d|}{3}\right)\pmod{p},\label{qb1}
\end{align}
where $\big(\frac{.}{p}\big)$ denotes the Legendre symbol.
The $a=1$ case of this congruence was first proved by Pan and Sun \cite[Theorem 1.2]{PSun} via a sophisticated combinatorial identity.

It is well known that binomial identities or congruences usually have nice $q$-analogues (see \cite{An2, An1, An, GZ1, Guo-ijnt,Guo-new, Guo-diff, GZ5,  GZ6, GL, GS1, GS2,  GZ, GuoZu3, Li, LW, LP, NP2, R1, WY0}).
John Green \cite{Gr} proved a conjecture made by Daan Krammer. The identity was served as an inspiration for searching $q$-analogues of some identities in \cite{SunR, SunR1}.
In particular, Tauraso \cite{GZ1} proved the following generalization of (\ref{qb1}):
\begin{align*}
\sum_{k=0}^{n-1}q^k \qbinom{2k}{k+d}{}\equiv
    \left(\frac{n-|d|}{3}\right)q^{3r(r+1)/2+|d|(2r+1)}\pmod{\Phi_n(q)}
\end{align*}
with $r=\lfloor2(n-|d|)/3\rfloor.$  Liu and Petrov \cite{LP} proved that, for any positive integer $n$,
\begin{align*}
\sum_{k=0}^{n-1}q^k \qbinom{2k}{k}{}\equiv
    \left(\frac{n}{3}\right)q^{(n^2-1)/3}\pmod{\Phi_n(q)^2},
\end{align*}
which was originally conjectured by Guo \cite{Guo-ijnt}.
Here and in what follows, the {\it $n$-th cyclotomic polynomial} is defined as
$$
\Phi_n(q):=\prod_{\substack{1\leq k\leq n\\ \gcd(n,k)=1}}(q-e^{2\pi\sqrt{-1}\cdot\frac{k}{n}}).
$$
It is well known that $\Phi_d(q)$ is an irreducible polynomial with integral coefficients, and
$$
[n]=\prod_{\substack{d\geq 2\\ d\mid n}}\Phi_d(q).
$$
So, for $d\geqslant 2$, $\Phi_d(q)$ divides $[n]$ if and only if $d$ divides $n$.

Furthermore, the $q$-binomial coefficient $\qbinom{n}{k}{}$ is defined as
\begin{align*}
\qbinom{n}{k}{}=\qbinom{n}{k}{q}=\begin{cases}\dfrac{(q;q)_n}{(q;q)_{k}(q,q)_{n-k}} &\text{if }0\leqslant k\leqslant n,\\[10pt]
0 &\text{otherwise },
\end{cases}
\end{align*}
where
$$
(x;q)_n=\begin{cases}
(1-x)(1-xq)\cdots(1-xq^{n-1}), &\text{if }n\geq 1,\\[5pt]
1, &\text{if }n=0.
\end{cases}
$$

Guo and Zeng \cite{GZ1}  proved that if  $n\geq 1$ and $d= 0,1,\ldots, n-1,$  then
\begin{align}\label{GuoZeng1}
&\sum_{k=0}^{n-1}q^{-k(k+3)/2}\qbinom{2k}{k+d}{}(-q^{k+1};q)_{n-1-k}\nonumber\\
&\equiv\begin{cases} (-1)^{(n+d-2)/2}q^{(5-(n-d+1)^2)/4}(1-q^{n-d})\pmod{\Phi_n(q)}&\text{if }n\equiv d\pmod{2},\\[5pt]
(-1)^{(n+d-1)/2}q^{(5-(n-d)^2)/4} \pmod{\Phi_n(q)}&\text{if }n\not\equiv d\pmod{2},
\end{cases}
\end{align}
and
\begin{align}\label{GuoZeng2}
&\sum_{k=0}^{n-1}q^{-k(k+3)/2}\qbinom{2k}{k+d}{}(-q^{k};q)_{n-1-k}\nonumber\\
&\equiv\begin{cases} (-1)^{(n+d-2)/2}q^{(9-(n-d+1)^2)/4}\frac{1-q^{n-d}}{1+q}\pmod{\Phi_n(q)}&\text{if }n\equiv d\pmod{2},\\[5pt]
(-1)^{(n+d-1)/2}q^{(5-(n-d)^2)/4}\pmod{\Phi_n(q)} &\text{if }n\not\equiv d\pmod{2}.
\end{cases}
\end{align}
Clearly, \eqref{GuoZeng1} and \eqref{GuoZeng2} are partial $q$-analogues of \eqref{qb3}.

In this paper, we shall consider some $q$-analogues of \eqref{qb3} and \eqref{qb4} as well as some variations of them, which are
different from \eqref{GuoZeng1} and \eqref{GuoZeng2}.
Our main results can be stated as follows.

\begin{theorem}\label{qbtheorem1}
For $n\geqslant 1$ and  $0\leqslant d\leqslant n-1$, we have
\begin{align}
&\sum_{k=0}^{n-1}q^k\qbinom{2k}{k+d}{}(-q^{k+1};q)_{n-1-k}\nonumber\\
&\quad=\sum_{\substack{k=0\\ k\equiv n-d \bmod 2}}^{n-d}(-1)^{(n-d-k)/2}q^{3(n^2+k^2-d^2)/4-3nk/2+n-k}\frac{(1-q^k)(1+q^{n-k+1})}{(1-q^{2n-k+1})(1+q^n)}\qbinom{2n}{k}{}\label{qb17},\\
&\sum_{k=0}^{n-1}q^k\qbinom{2k}{k+d}{}(-q^{k+1};q)_{n-1-k}^2\nonumber\\
&\quad=\sum_{k=d+1}^{n}\sum_{j=k+1}^{n+1}\frac{q^{(j^2-3j+k^2-k)/2-d^2+1}1+q^{j-1}}{1+q^n}\qbinom{2n}{n-j+1}{}\label{qb19}.
\end{align}
\end{theorem}

\begin{corollary}\label{qbinomailmain1}
Let $n\geqslant 1$ and  $0\leqslant d\leqslant n-1$. Then, modulo $\Phi_{n}(q)$, we have
\begin{align}\label{qb6}
\sum_{k=0}^{n-1}q^k\qbinom{2k}{k+d}{}(-q^{k+1};q)_{n-1-k}
\equiv\begin{cases} 0&\!\!\!\text{if }n\equiv d\pmod{2},\\
q^{3(n-1)^2/4-3d^2/4-1} &\!\!\!\text{if }n\equiv d+1\pmod{4},\\
-q^{3(n-1)^2/4-3d^2/4-1} &\!\!\!\text{if }n\equiv d-1\pmod{4},
\end{cases}
\end{align}
and
\begin{align}\label{qb7}
\sum_{k=0}^{n-1}q^k\qbinom{2k}{k+d}{}(-q^{k+1};q)_{n-1-k}^2
\equiv \sum_{k=d+1}^{n}q^{\binom{n}{2}+\binom{k}{2}-d^2}.
\end{align}

\end{corollary}
Note that (\ref{qb6}) was proved by Guo and Zeng \cite[Corollary 4.2]{GZ1} by using Andrews's $q$-analogue of Gauss's second theorem (see \cite{An2, An}).
\eqref{qb6} and \eqref{qb7} may be deemed  partial $q$-analogues of \eqref{qb3} for the cases $m=2$ and $m=4$, respectively.
Thus, when $n=p^a, q\rightarrow 1 $, the $q$-congruences \eqref{qb6} and \eqref{qb7}  reduce to the following consequences \cite[Corollary 1.1]{SunR}.
\begin{corollary} \label{qbinomailmain1'}
For any odd prime $p$ and $0\leqslant d\leqslant p^a-1$ with $a\in \Z^{+}$,  modulo $p$, we have
\begin{align}\label{qb8}
\sum_{k=0}^{p^a-1}\binom{2k}{k+d}\frac{1}{2^k}\equiv\begin{cases} 0&\text{if }p^a\equiv d\pmod{2},\\
1 &\text{if }p^a\equiv d+1\pmod{4},\\
-1 &\text{if }p^a\equiv d-1\pmod{4},
\end{cases}
\end{align}
and
\begin{align}\label{qb8'}
\sum_{k=0}^{p^a-1}\binom{2k}{k+d}\frac{1}{4^k}\equiv -d.
\end{align}
\end{corollary}

\begin{theorem}\label{qbtheorem3}
For $n\geqslant 1$ and  $1\leqslant d\leqslant n-1$, we have
\begin{align}
&\sum_{k=1}^{n-1}q^k\qbinom{2k}{k+d}{}\frac{(-q^{k+1};q)_{n-k}}{[k]}\nonumber\\
&=-\sum_{k=1}^{\lfloor(n+1-d)/2\rfloor}(-1)^kq^{3k^2+(3d-5)k-2d+2}\frac{[2d+4k-2]}{[d][n]}\qbinom{2n}{n-d-2k+1}{}\label{qb20},\\
&\sum_{k=1}^{n}q^k\qbinom{2k}{k+d}{}\frac{(-q^{k+1};q)_{n-k}^2}{[k]}\nonumber\\
&=\sum_{k=d}^{n-1}q^{\binom{k+1}{2}-\binom{d}{2}}\frac{\qbinom{2n+1}{n-k}{}}{[d]}+\frac{q^{\binom{n+1}{2}-\binom{d}{2}}}{[d]}\label{qb21}.
\end{align}
\end{theorem}

\begin{corollary}\label{qbinomailmain2}
For $n\geqslant 1$ and  $1\leqslant d\leqslant n-1$, modulo $\Phi_n(q)$, we have
\begin{align}\label{qb9}
&[d]\sum_{k=1}^{n-1}q^k\qbinom{2k}{k+d}{}\frac{(-q^{k+1};q)_{n-k}}{[k]}\nonumber\\
&\equiv\begin{cases} -2q^{-\binom{n+1}{2}-\binom{d+1}{2}}\sum_{k=2}^{(n+1-d)/2-1}(-1)^kq^{k^2+(d-2)k+1}(1+q^{d+2k-1})\\
-2(-1)^{(n+1-d)/2}q^{-(n^2+3d^2-2d-1)/4}+2q^{-\binom{n}{2}-\binom{d}{2}}(1+q^{d+1})&\text{if }n\not\equiv d\pmod{2},\\[5pt]
2q^{-\binom{n+1}{2}-\binom{d+1}{2}}\sum_{k=2}^{(n-d)/2-1}(-1)^kq^{k^2+(d-2)k+1}(1+q^{d+2k-1})\\
+2(-1)^{(n-d)/2}q^{(-n^2-2n-3d^2+2d)/4}(1+q)-2q^{-\binom{n}{2}-\binom{d}{2}}(1+q^{d+1}) &\text{if }n\equiv d\pmod{2},\\
\end{cases}
\end{align}
and
\begin{align}\label{qb10}
[d]\sum_{k=1}^{n}q^k\qbinom{2k}{k+d}{}\frac{(-q^{k+1};q)_{n-k}^2}{[k]}
\equiv \begin{cases} 2q^{-\binom{d}{2}} &\text{if } n \equiv 1\pmod{2},\\
-2q^{-\binom{d}{2}} &\text{otherwise }.
\end{cases}
\end{align}

\end{corollary}

Once again, the congruences \eqref{qb9} and \eqref{qb10} may be deemed  partial $q$-analogues of \eqref{qb4} for the cases $m=2, 4$ which imply the following conclusion.
\begin{corollary}\label{qbinomailmain2'}
For any odd prime $p$ and $1\leqslant d\leqslant p^a-1$ with $a\in \Z^{+}$,  modulo $p$, we have
\begin{align}\label{qb11}
d\sum_{k=1}^{p^a-1}\frac{\binom{2k}{k+d}}{k2^k}-(-1)^d\equiv\begin{cases} 0 &\text{if } p^a\not\equiv d\pmod{2},\\
1 &\text{if } p^a\equiv d\pmod{4},\\
-1 &\text{if } p^a\equiv d+2\pmod{4},\\
\end{cases}
\end{align}
and
\begin{align}\label{qb12}
d\sum_{k=1}^{p^a}\frac{\binom{2k}{k+d}}{k4^k}\equiv\begin{cases} 1/2 &\text{if } p^a\not\equiv d\pmod{2},\\
-1/2 &\text{if } p^a\equiv d\pmod{2}.\\
\end{cases}
\end{align}
\end{corollary}

The rest of this paper is organized as follows. In the next section, we are going to present some auxiliary lemmas and prove
Theorem \ref{qbtheorem1} and Corollary \ref{qbinomailmain1}. Finally, Theorem \ref{qbtheorem3} and Corollary \ref{qbinomailmain2} will be proved in Section 3.

\section{Proofs of Theorem \ref{qbtheorem1} and Corollary \ref{qbinomailmain1}}
\setcounter{lemma}{0}
\setcounter{theorem}{0}
\setcounter{corollary}{0}
\setcounter{remark}{0}
\setcounter{equation}{0}
\setcounter{conjecture}{0}

Firstly, we give few properties that we will need later.

\begin{lemma}{\rm\cite{Ca,Fr,HS}}\label{qblem1}
Let $a,b,c,d\in \Z$ and $0\leq b,d \leq n-1$. Then, modulo $\Phi_n(q)$,
\begin{align}\label{qb13}
\qbinom{an+b}{cn+d}{}\equiv \binom{a}{c}\qbinom{b}{d}{}.
\end{align}
\end{lemma}

Note that \eqref{qblem1} is a generation of the well-known Lucas congruence. From  \eqref{qblem1}
we can easily deduce the following result.

\begin{lemma}{\rm\cite{R1}}\label{qblem3}
Let $n>1$, and $a\geqslant 1$. Then,  modulo $\Phi_n(q)$,
\begin{align}\label{qb15}
\qbinom{an}{k}{}\equiv \begin{cases} \binom{a}{k/n} &\text{if }n \mid k,\\
0 &\text{otherwise },
\end{cases}
\end{align}
and
\begin{align}\label{qb16}
\qbinom{2n+1}{k}{}\equiv \begin{cases} 1 &\text{if } k=0,1,2n,2n+1, \\
2&\text{if } k=n, n+1,\\
0 &\text{otherwise }.
\end{cases}
\end{align}
\end{lemma}

\begin{proof}[Proof of Theorem \ref{qbtheorem1}]
Denote the left-hand side and right-hand side of \eqref{qb17} by $S(n,d)$ and $T(n,d)$, respectively.

For $n\geq 1,$ we have
\begin{align*}
 &S(n,n)= T(n,n)=0, S(n,n-1)= T(n,n-1)=q^{n-1},\\
 &S(n,n-2)= T(n,n-2)=\begin{cases}q^{n-1}[2n-2]+q^{n-2}(1+q^{n-1})&\text{if }n> 1, \\
 0 &\text{if }n = 1.\end{cases}
\end{align*}
Next, we shall prove that for $0\leq d\leq n-1,$ both $S(n,d)$ and $T(n,d)$ satisfy the following recurrence equation
\begin{align}\label{qb18}
X(n,d)+q^{3d+3}X(n,d+2)=\frac{(q^d-q^n)(1+q^{d+1})}{(1-q^{n+d+1})(1+q^n)}\qbinom{2n}{n+d}{}.
\end{align}

We start with $S(n,d),$ and shall prove (\ref{qb18}) by induction on $n$. Clearly, it holds for $n=1,$ and we will prove that, for $n\geq 1,$
$$S(n+1,d)+q^{3d+3}S(n+1,d+2)=\frac{(q^d-q^{n+1})(1+q^{d+1})}{(1-q^{n+d+2})(1+q^{n+1})}\qbinom{2n+2}{n+d+1}{}.$$
Since the left-hand side is equal to
$$ (1+q^n)S(n,d)+q^n\qbinom{2n}{n+d}{}+q^{3d+3}\left((1+q^n)S(n,d+2)+q^n\qbinom{2n}{n+d+2}{}\right),$$
by induction, we only need to check that
\begin{align*}
&\frac{(q^d-q^n)(1+q^{d+1})}{1-q^{n+d+1}}\qbinom{2n}{n+d}{}+q^n\qbinom{2n}{n+d}{}+q^{3d+3+n}\qbinom{2n}{n+d+2}{}\\
&\quad=\frac{(q^d-q^{n+1})(1+q^{d+1})}{(1-q^{n+d+2})(1+q^{n+1})}\qbinom{2n+2}{n+d+1}{},
\end{align*}
which is indeed true.

For $T(n,d)$, there holds
\begin{align*}
&q^{3d+3}T(n,d+2)\\
=&q^{3d+3}\sum_{\substack{k=0\\ k\equiv n-d \bmod 2}}^{n-d-2}(-1)^{(n-d-k-2)/2}q^{3(n^2+k^2-(d+2)^2)/4-3nk/2+n-k}\frac{(1-q^k)(1+q^{n-k+1})}{(1-q^{2n-k+1})(1+q^n)}\qbinom{2n}{k}{} \\
=&-\sum_{\substack{k=0\\ k\equiv n-d \bmod 2}}^{n-d-2}(-1)^{(n-d-k)/2}q^{3(n^2+k^2-d^2)/4-3nk/2+n-k}\frac{(1-q^k)(1+q^{n-k+1})}{(1-q^{2n-k+1})(1+q^n)}\qbinom{2n}{k}{}\\
=&-T(n,d)+\frac{(q^d-q^n)(1+q^{d+1})}{(1-q^{n+d+1})(1+q^n)}\qbinom{2n}{n+d}{}.
\end{align*}
Hence, the identity \eqref{qb17} is concluded.

Now, let $S(n,d)$ and $T(n,d)$ denote the left-hand side and right-hand side of \eqref{qb19}, respectively.
Likewise, we need to check that for $0\leq d\leq n-1, S(n,d)=T(n,d).$
For $n\geq 1$ we have
\begin{align*}
&S(n,n)= T(n,n)=0, S(n,n-1)= T(n,n-1)=q^{n-1},\\
&S(n,n-2)= T(n,n-2)=\begin{cases}q^{n-1}[2n-2]+q^{n-2}(1+q^{n-1})^2&\text{if }n > 1,\\0 &\text{if }n = 1.\end{cases}
\end{align*}
Then, we will show that  both $S(n,d)$ and $T(n,d)$ satisfy the following recurrence equation
$$X(n,d)-q^{d}(1+q^{d+1})X(n,d+1)+q^{3d+3}X(n,d+2)=\frac{q^d(1-q^n)(1+q^{d+1})(q;q)_{2n-1}}{(q;q)_{n-d-1}(q;q)_{n+d+1}}.$$

We start with $S(n,d)$. Clearly, the above identity holds for $n=1$. We need to prove that
\begin{align}
&S(n+1,d)-q^{d}(1+q^{d+1})S(n+1,d+1)+q^{3d+3}X(n+1,d+2) \label{eq:aa-aa}\\
&\quad=\frac{q^d(1-q^{n+1})(1+q^{d+1})(q;q)_{2n+1}}{(q;q)_{n-d}(q;q)_{n+d+2}} \notag
\end{align}
for $n\geq 2$.
The left-hand side of \eqref{eq:aa-aa} is equal to
\begin{align*}
&(1+q^n)^2S(n,d)+q^n\qbinom{2n}{n+d}{}-q^{d}(1+q^{d+1})\left((1+q^n)^2S(n,d+1)+q^n\qbinom{2n}{n+d+1}{}\right)\\
&\quad+q^{3d+3}\left((1+q^n)^2S(n,d+2)+q^n\qbinom{2n}{n+d+2}{}\right).
\end{align*}
By induction, we only need to verify that
\begin{align*}
&(1+q^n)^2\frac{q^d(1-q^n)(1+q^{d+1})(q;q)_{2n-1}}{(q;q)_{n-d-1}(q;q)_{n+d+1}}+q^n\qbinom{2n}{n+d}{}-q^{d+n}(1+q^{d+1})\qbinom{2n}{n+d+1}{}\\
&\quad+q^{3d+3+n}\qbinom{2n}{n+d+2}{}=\frac{q^d(1-q^{n+1})(1+q^{d+1})(q;q)_{2n+1}}{(q;q)_{n-d}(q;q)_{n+d+2}},
\end{align*}
which is clearly true.

Now, let us turn to consider $T(n,d),$
\begin{align*}
&-q^{d}(1+q^{d+1})T(n,d+1)\\
=&-(1+q^{d+1})\sum_{k=d+2}^{n}\sum_{j=k+1}^{n+1}q^{(j^2-3j+k^2-k)/2-d^2-d}\frac{1+q^{j-1}}{1+q^n}\qbinom{2n}{n-j+1}{}\\
=&-\sum_{k=d+2}^{n}\sum_{j=k+1}^{n+1}q^{(j^2-3j+k^2-k)/2-d^2-d}\frac{1+q^{j-1}}{1+q^n}\qbinom{2n}{n-j+1}{}\\
&\quad-q^{d+1}\sum_{k=d+2}^{n+1}\sum_{j=k+1}^{n+1}q^{(j^2-3j+k^2-k)/2-d^2-d}\frac{1+q^{j-1}}{1+q^n}\qbinom{2n}{n-j+1}{}\\
=&-q^{3d+3}T(n,d+2)-q^{3d+3}\sum_{j=d+3}^{n+1}q^{(j^2-3j+(d+2)^2-(d+2))/2-(d+2)^2+1}\frac{1+q^{j-1}}{1+q^n}\qbinom{2n}{n-j+1}{}\\
&\quad -T(n,d)+\sum_{j=d+2}^{n+1}q^{(j^2-3j+(d+1)^2-(d+1))/2-d^2+1}\frac{1+q^{j-1}}{1+q^n}\qbinom{2n}{n-j+1}{}\\
=&-q^{3d+3}T(n,d+2)-T(n,d)+q^d\frac{(1-q^n)(1+q^{d+1})(q;q)_{2n-1}}{(q;q)_{n-d-1}(q;q)_{n+d+1}}.
\end{align*}
Therefore, the identity \eqref{qb19} follows.
\end{proof}

\begin{proof}[Proof of Corollary \ref{qbinomailmain1}]
For $d=0,1,2\ldots,$ via \eqref{qb15}, we have
\begin{align*}
\frac{1-q^k}{1-q^{2n-k+1}}\qbinom{2n}{k}{}=\frac{[2n]}{[2n-k+1]}\qbinom{2n-1}{2n-k}{}=\qbinom{2n}{k-1}{},
\end{align*}
and
\begin{align*}
&\sum_{k=0}^{n-1}q^k\qbinom{2k}{k+d}{}(-q^{k+1};q)_{n-1-k}\\
&\equiv1/2\sum_{k=1}^{n-d}(-1)^{(n-d-k)/2}q^{3(n^2+k^2-d^2)/4-3nk/2-k}[2\mid n-d+k](1+q^{1-k})\qbinom{2n}{k-1}{}\\
&\equiv\begin{cases} 0&\text{if }n\equiv d\pmod{2}\\
(-1)^{(n-1-d)/2}q^{3(n-1)^2/4-3d^2/4-1} &\text{if }n\not\equiv d\pmod{2}
\end{cases}  \pmod{\Phi_n(q)}.
\end{align*}
Similarly, since $$\qbinom{2n}{n-j+1}{}\equiv 0\pmod{\Phi_n(q)},$$ for $d+1\leq k<j< n+1,$
we are led to
\begin{align*}
\sum_{k=0}^{n-1}q^k\qbinom{2k}{k+d}{}(-q^{k+1};q)_{n-1-k}^2
\equiv\sum_{k=d+1}^{n}q^{\binom{n}{2}+\binom{k}{2}-d^2}\pmod{\Phi_n(q)}.
\end{align*}
The proof of the theorem is complete.
\end{proof}

\section{Proofs of Theorem \ref{qbtheorem3} and Corollary \ref{qbinomailmain2}}
\setcounter{lemma}{0}
\setcounter{theorem}{0}
\setcounter{corollary}{0}
\setcounter{remark}{0}
\setcounter{equation}{0}
\setcounter{conjecture}{0}

\begin{proof}[Proof of Theorem \ref{qbtheorem3}]
Let $S(n,d)$ and $T(n,d)$ be the left-hand side and the right-hand side of the claimed identity.
It is very easy to verify that  $S(n,d)=T(n,d)$ holds for $d=n,n-1,n-2.$
Along the same lines of the proof of Theorem \ref{qbtheorem1}, we can prove that
\begin{align*}
&\frac{q^{2+3d}(1-q^{d+2})}{1-q^d}T(n,d+2)=-T(n,d)+q^d\frac{(1+q^n)(1-q^{2d+2})(q;q)_{2n-1}}{(1-q^d)(q;q)_{n+d+1}(q^2;q)_{n-d-2}},
\end{align*}
where $\lfloor x\rfloor$ denotes the largest integer less than or equal to $x$.

Moreover, for $n\geq 1,$ we shall prove that, for $1\leq d\leq n,$  $S(n,d)$  satisfy the same recurrence equation.
Clearly, it holds for $n=1$. We will prove that, for $n\geq 1$,
\begin{align}
&\frac{q^{2+3d}(1-q^{d+2})}{1-q^d}S(n+1,d+2)+S(n+1,d)=q^d\frac{(1+q^{n+1})(1-q^{2d+2})(q;q)_{2n+1}}{(1-q^d)(q;q)_{n+d+2}(q^2;q)_{n-d-1}}.
\label{eq:cc-cc}
\end{align}
It is clear to see that the left-hand side is equal to
\begin{align*}
&(1+q^{n+1})S(n,d)+q^n\frac{(1+q^{n+1})\qbinom{2n}{n+d}{}}{[n]}\\
&\quad+\frac{q^{2+3d}(1+q^{n+1})(1-q^{d+2})}{1-q^d}S(n,d+2)+\frac{q^{2+3d+n}(1-q^{d+2})(1+q^{n+1})\qbinom{2n}{n+d+2}{}}{(1-q^d)[n]}.
\end{align*}
By induction and using the following identity
\begin{align*}
&q^d\frac{(1+q^n)(1-q^{2d+2})(q;q)_{2n-1}}{(1-q^d)(q;q)_{n+d+1}(q^2;q)_{n-d-2}}+\frac{q^{n+2+3d}(1-q^{d+2})\qbinom{2n}{n+d+2}{}}{(1-q^d)[n]}+q^n\frac{\qbinom{2n}{n+d}{}}{[n]}\\
=&q^d\frac{(1-q^{2d+2})(q;q)_{2n+1}}{(1-q^d)(q;q)_{n+d+2}(q^2;q)_{n-d-1}},
\end{align*}
we complete the proof of \eqref{eq:cc-cc}.

Similarly, we denonte the left-hand side and the right-hand side of \eqref{qb21} by $S(n,d)$ and $T(n,d)$, respectively. Clearly,  $S(n,d)=T(n,d)$  holds for $d=n$.
Moreover,
\begin{align*}
T(n,d+1)=&\sum_{k=d+1}^{n-1}q^{\binom{k+1}{2}-\binom{d+1}{2}}\frac{\qbinom{2n+1}{n-k}{}}{[d+1]}+\frac{q^{\binom{n+1}{2}-\binom{d+1}{2}}}{[d+1]}\\
=&-\frac{1-q^{-d}}{1-q^{d+1}}T(n,d)-\frac{\qbinom{2n+1}{n-d}{}}{[d+1]}.
\end{align*}
It remains to show that, for $n\geq 1$ and $1\leq d\leq n$,
\begin{align*}
S(n,d+1)+\frac{1-q^{-d}}{1-q^{d+1}}S(n,d)=-\frac{\qbinom{2n+1}{n-d}{}}{[d+1]}.
\end{align*}
We proceed by induction on $n$. Clearly, it holds for $n=1$. Suppose that it holds for $n$. We will prove that
\begin{align}
&S(n+1,d+1)+\frac{1-q^{-d}}{1-q^{d+1}}S(n+1,d)=-\frac{\qbinom{2n+3}{n+1-d}{}}{[d+1]}. \label{eq:dd-dd}
\end{align}
The left-hand side of the above identity is equal to
\begin{align*}
&\frac{1-q^{-d}}{1-q^{d+1}}(1+q^{n+1})^2S(n,d)+(1+q^{n+1})^2S(n,d+1)\\
&\quad+\frac{q^{n+1}(1-q^{-d})}{(1-q^{d+1})[n+1]}\qbinom{2n+2}{n+1+d}{}+q^{n+1}\frac{\qbinom{2n+2}{n+2+d}{}}{[n+1]}.
\end{align*}
By the induction hypothesis and applying the following identity
\begin{align*}
(1+q^{n+1})^2\frac{\qbinom{2n+1}{n-d}{}}{[d+1]}-\frac{\qbinom{2n+3}{n+1-d}{}}{[d+1]}-q^{n+1}\frac{(1-q^{-d})\qbinom{2n+2}{n+1+d}{}}{(1-q^{d+1})[n+1]}=q^{n+1}\frac{\qbinom{2n+2}{n+2+d}{}}{[n+1]},
\end{align*}
we are led to \eqref{eq:dd-dd}. This completes the proof.
\end{proof}

\begin{proof}[Proof of Corollary \ref{qbinomailmain2}]
Notice that for $1\leq j<n,$
$$\qbinom{an}{j}{}=\frac{[an]}{[j]}\qbinom{an-1}{j-1}{}\equiv 0\pmod{\Phi_n(q)},$$
and
$$ \qbinom{an-1}{j-1}q\equiv\qbinom{-1}{j-1}q=(-1)^{j-1} q^{-\binom{j}{2}}\pmod{\Phi_n(q)}.
$$
When $n-d\equiv 0\pmod{2}$, we have
\begin{align*}
&[d]\sum_{k=1}^{n-1}q^k\qbinom{2k}{k+d}{}\frac{(-q^{k+1};q)_{n-k}}{[k]}\nonumber\\
&\quad=-\sum_{k=1}^{(n-d)/2}(-1)^kq^{3k^2+(3d-5)k-2d+2}\frac{[2d+4k-2]}{[n]}\qbinom{2n}{n-d-2k+1}{}\\
&\quad\equiv -(-1)^{(n-d)/2}q^{3(n-d)^2/4+(3d-5)(n-d)/2-2d+2}(1+q^n)[2n-2]+q^d\qbinom{2n}{n-d-1}{}\frac{[2d+2]}{[n]}\\
&\qquad+\sum_{k=2}^{(n-d)/2-1}(-1)^kq^{3k^2+(3d-3)k-d+1}(1+q^n)(1+q^{d+2k-1})\qbinom{-1}{n-d-2k}{}\\
&\quad\equiv-2(-1)^{(n-d)/2}q^{(3n^2-10n-3d^2+2d+8)/4}[2n-2]-2q^{-(n^2-2nd-3n+d^2-d)/2}(1+q^{d+1})\\
&\qquad+2\sum_{k=2}^{(n-d)/2-1}(-1)^kq^{-\binom{n+1}{2}-\binom{d+1}{2}+k^2+(d-2)k+1}(1+q^{d+2k-1})\\
&\quad\equiv 2q^{-\binom{n+1}{2}-\binom{d+1}{2}}\sum_{k=2}^{(n-d)/2-1}(-1)^kq^{k^2+(d-2)k+1}(1+q^{d+2k-1})\\
&\qquad+2(-1)^{(n-d)/2}q^{(-n^2-2n-3d^2+2d)/4}(1+q)-2q^{-\binom{n}{2}-\binom{d}{2}}(1+q^{d+1})\pmod{\Phi_n(q)}.
\end{align*}

When $n-d\equiv 1\pmod{2}$, we obtain
\begin{align*}
&[d]\sum_{k=1}^{n-1}q^k\qbinom{2k}{k+d}{}\frac{(-q^{k+1};q)_{n-k}}{[k]}\nonumber\\
&\quad=-\sum_{k=1}^{(n-d+1)/2}(-1)^kq^{3k^2+(3d-5)k-2d+2}\frac{[2d+4k-2]}{[n]}\qbinom{2n}{n-d-2k+1}{}\\
&\quad\equiv -2(-1)^{(n-d+1)/2}q^{3(n-d+1)^2/4+(3d-5)(n-d+1)/2-2d+2}+q^d\qbinom{2n}{n-d-1}{}\frac{[2d+2]}{[n]}\\
&\qquad+2\sum_{k=2}^{(n-d+1)/2-1}(-1)^kq^{3k^2+(3d-3)k-d+1}(1+q^{d+2k-1})\qbinom{-1}{n-d-2k}{}\\
&\quad\equiv-2q^{-\binom{n+1}{2}-\binom{d+1}{2}}\sum_{k=2}^{(n-d+1)/2-1}(-1)^kq^{k^2+(d-2)k+1}(1+q^{d+2k-1})\\
&\qquad-2(-1)^{(n-d+1)/2}q^{-(n^2+3d^2-2d-1)/4}+2q^{-\binom{n}{2}-\binom{d}{2}}(1+q^{d+1})\pmod{\Phi_n(q)}.\\
\end{align*}
Thus, the $q$-congruence \eqref{qb9} is concluded.

Similarly as before, by Lemma \eqref{qblem3},  we have
\begin{align*}
[d]\sum_{k=1}^{n}q^k\qbinom{2k}{k+d}{}\frac{(-q^{k+1};q)_{n-k}^2}{[k]}
=&\sum_{k=d}^{n-1}q^{\binom{k+1}{2}-\binom{d}{2}}\qbinom{2n+1}{n-k}{}+q^{\binom{n+1}{2}-\binom{d}{2}}\\
\equiv&q^{\binom{n}{2}-\binom{d}{2}}+q^{\binom{n+1}{2}-\binom{d}{2}}\\
\equiv& \begin{cases} 2q^{-\binom{d}{2}}\pmod{\Phi_n(q)} &\text{if  $n$ is odd},\\
-2q^{-\binom{d}{2}}\pmod{\Phi_n(q)} &\text{otherwise }.
\end{cases}
\end{align*}
This completes the proof of Corollary \eqref{qbinomailmain2}.
\end{proof}

\end{document}